\documentclass[12pt,a4paper]{article}
\usepackage[a4paper]{geometry}
\usepackage[english]{babel}
\usepackage{amsfonts}
\usepackage{amsmath,amsthm}
\usepackage{eucal}
\usepackage{latexsym}
\usepackage{amssymb}
\usepackage{amsbsy}
\usepackage{amsmath}
\usepackage{amscd}
\usepackage{amsgen}
\usepackage{amsopn}
\usepackage{amstext}
\usepackage{ifthen}
\usepackage{amsxtra}
\usepackage{amsfonts}
\usepackage{float}
\usepackage{graphicx}
\usepackage{cite}
\newtheorem{theorem}{Theorem}[section]
\newtheorem{lemma}[theorem]{Lemma}

\newtheorem{proposition}[theorem]{Proposition }
\newtheorem{definition}[theorem]{Definition }
\newtheorem{corollary}[theorem]{Corollary}

\newenvironment{proof of lemma 5.2}[1][\bf{Proof of Lemma 5.2.~~}]{\noindent{\it{#1}}}{\hfill$\square$\\ }
\begin{document}
\begin{center}
{\Large
          { \bf  On the ascent-descent spectrum}

}

\end{center}
\begin{center}
{\large
          {\bf
 Nassim Athmouni$^{a}$, Mondher Damak$^{b}$, Chiraz Jendoubi$^{c}$
}

}

\begin{center}
{$^a$ Faculty of Sciences ,
University of Gafsa,
Gafsa-Tunisia.\\
E-mail: athmouninassim@yahoo.fr
}
\end{center}
\begin{center}
{$^{b,c}$ Faculty of Sciences,
University of Sfax,
Sfax-Tunisia.\\
$^b$ E-mail: mondher$\_$damak@yahoo.com\\
$^{c}$ E-mail: chirazjendoubi@live.fr
}
\end{center}
\end{center}
\begin{abstract}
 We  establish the various properties as well as diverse relations of the ascent and descent spectra for bounded linear operators. We specially focus on the theory of subspectrum. Furthermore, we construct a new concept of convergence for such spectra.
\end{abstract}
\section{Introduction and  main results}
 Denote by $B(X)$ the algebra of bounded linear operators in a Banach space  $X$ and by $\mathcal{F}(X)$ the set of finite rank operators on $X$. For $T\in B(X)$, we use $N(T)$ and $R(T)$ respectively to denote the null-space and the range of $T$. Let $\alpha(T)=dim N(T)$, if $N(T)$ is finite dimensional and let $\alpha(T)=\infty$ if  $N(T)$ is infinite dimensional. Similarly, let $\beta(T)=dim X / R(T)=codim R(T)$,  if $ X / R(T)$ is finite dimensional and let $\beta(T)=\infty$ if $X / R(T)$ is infinite dimensional.\\
\par The ascent of $T\in B(X)$, denoted by $asc(T)$, is the smallest $n\in\mathbb{N}$ satisfying
 $N(T^n)=N(T^{n+1})$. If such $n$ does not exist then $asc(T)=\infty$. The descent of $T$, denoted by $dsc(T)$, is the smallest $n\in\mathbb{N}$ satisfying $R(T^n)=R(T^{n+1})$. If such $n$ does not exist then $dsc(T)=\infty$. Next, we denote by $Asc(B(X))$ the space of bounded operators $T$ such that $asc(T)$ is finite and by $Dsc(B(X))$ the space of bounded operators $T$ such that $dsc(T)$ is finite. It is worth noting that this algebraic theory was mostly developed by M. A. Kaashoek \cite{20-int} and A.E. Taylor \cite{8-int}. As an interesting result which characterizes ascent-descent operators is the following proposition :
\begin{proposition}\label{cont} \cite{54, burgos} Let $T$ be a linear operator on a vector space $X$ and $m$ be a positive natural. The following assertions hold true :\\
 i) $asc(T)\leq m<\infty$ if and only if for every $n\in\mathbb{N}$,  $R(T^m)\cap N( T^n)=\{0\}$.\\
ii) $dsc(T)\leq m<\infty$ if and only if for every $n\in\mathbb{N}$ there exists a subspace $Y_n\subset N(T^m)$ satisfying $X=Y_n\oplus R(T^n)$.
\end{proposition}
Recently, stability problems of operators under perturbation have attracted many researchers and undergone important contributions, see for instance \cite{ABJ, 24-int, 5-boasso, burgos, mnif, monsieur, vladimir, 31-int}. Recall that, in \cite{zdr}, authors have proved nice relations between the left Browder spectrum and the left invertible spectrum (respectively, between the right Browder spectrum and the right invertible spectrum). Motivated by these last works and considering the following subsets :
  \begin{eqnarray*}
&& \sigma_{asc}(T)=\{\lambda\in\mathbb{C}: T-\lambda \notin Asc( B(X)) \}\hbox{-the ascent spectrum},\\
&& \sigma_{dsc}(T)=\{\lambda\in\mathbb{C}: T-\lambda \notin Dsc( B(X)) \}\hbox{-the descent spectrum}.
\end{eqnarray*}
\par Focusing on the axiomatic theory of subspectrum which was introduced by Slodowski and Zelazko \cite{17, 21}, we, first, recall the following definitions.
  \begin{definition} \label{Saber}
  A subspectrum $\tilde{\sigma}$ in $B(X)$ is a mapping which assigns to every $n-$tuple $(T_1,..,T_n)$ of mutually commuting elements of $B(X)$ a non-empty compact subset $\tilde{\sigma}(T_1,...,T_n)\subset\mathbb{C}^n$ such that\\
  $i)$ $\tilde{\sigma}(T_1,...,T_n)\subset\sigma(T_1)\times..\times\sigma(T_n)$,\\
  $ii)$ $\tilde{\sigma}(p(T_1,...,T_n))=p(\tilde{\sigma}(T_1,...,T_n))$ for every commuting $T_1,..,T_n\in B(X)$ and every polynomial mapping $p=(p_1,..,p_m):\mathbb{C}^n\to\mathbb{C}^m$.
\end{definition}
The concept of Definition \ref{Saber} is trully suitable since it comprises for example the left (right) spectrum, the left (right) approximate point spectrum, the Harte (the union of the left and right) spectrum. \ \ Despite, there are also many examples of spectrum, frequently characterized only for single elements, which are not covered by the axiomatic theory of $\dot{Z}$elazko. 
\begin{definition}\cite{km}
Let $\mathcal{A}$ be a Banach algebra. A non-empty subset $R$ of $\mathcal{A}$ is called a regularity if\\
$i)$ if $a\in\mathcal{A}$ and $n\in\mathbb{N}$ then $a\in\mathcal{A}\Leftrightarrow a^n\in\mathcal{A}$,\\
$ii)$ if $a,\, b,\, c,\, d $ are mutually commuting elements of $\mathcal{A}$ and $ac+bd=1_{\mathcal{A}}$, then
$ab\in R \Leftrightarrow a\in R \hbox{ and } b\in R$.
\end{definition}
Such theory was taken into account by many researchers due to its widespread concept that generalizes spectra for single elements to spectra for n-tuple elements. We recall that, in \cite{anar}, the author introduced regularities and subspectra in a unital noncommutative Banach algebra and showed that there is a correspondence between them  similarly to the commutative case. In \cite{km}, the authors gave the following conditions on a regularity $R$, to extend spectra to subspectra.\\
    $(C_1)$ $ab\in R \Leftrightarrow a\in R$  and  $b\in R$  for all commuting elements $a,\, b\in\mathcal{A}$.\\
 $(C_2)$ "Continuity on commuting elements", i.e :
 If $a_n,\, a\in\mathcal{A}$, $a_n$ converges to  $a$ and $a_na=aa_n$ for every $n$ then $\lambda\in\sigma_R(a)$ if and only if there exists a sequence $\lambda_n\in\sigma_R(a_n)$ such that $\lambda_n$ converges to $\lambda$.\\
 It is worth noting that the space of bounded linear operators $B(X)$ is a particular case of the Banach algebra $\mathcal{A}$ and $Asc(B(X))$ and $Dsc(B(X))$ are regularities on $B(X)$. For more information, see \cite{5-boasso, 6-boasso, 19-boasso}. Throughout our work,under suitable optimal hypothesis, we will prove that conditions $(C_1)$ and $(C_2)$ are approved properties for the ascent and descent spectra. For the  condition $(C_1)$ ,  we need the following assumptions.\\
  $(H_1)$: \ \ \ $ST=TS\ \hbox{ and }  \forall p\in\mathbb{N},\   \,N((TS)^p)=N(T^p)\oplus N(S^p).$\\
  $(H_2)$:  \ \ \  $ST \hbox{ is with finite descent }  n_0 \hbox{ and  } N(S^{n_0})\subseteq R(T) \hbox{ or }  N(T^{n_0})\subseteq R(S).$\\
  Consider the following subsets :
  \begin{eqnarray*}
  &&\tilde{F}:=\{F\in B(X)\hbox{such that there exists $n_0\in\mathbb{N}$ for which }F^{n_0}\in\mathcal{F}(X)\},\\
  &&\mathcal{R}:=\{\lambda\in\mathbb{C} \hbox{ such that if $S+T-\lambda$ is ascent then $(S-\lambda)$ and $(T-\lambda)$ do not satisfy $(H_1)$}\},\\
  &&\mathcal{M}:=\bigg\{\lambda\in\mathbb{C} \hbox{ such that if $dsc((S-\lambda)(T-\lambda))=n_0$ then $codim (R(S-\lambda)^{n_0})$ is infinite},\\
  &&\hbox{ or $codim (R(T-\lambda)^{n_0})$ is infinite}\bigg\},\\
  &&\mathcal{N}:=\big\{\lambda\in\mathbb{C} \hbox{ such that if $S+T-\lambda$ is descent then $(S-\lambda)$ and $(T-\lambda)$ do not satisfy $(H_1)$ or $(H_2)$}\big\}.
  \end{eqnarray*}
  Our main results  assuring condition $(C_1)$ reside in :
   \begin{theorem}\label{th1}
   Assume that $S$ and $T$ be two bounded linear operators satisfying $(H_1)$ and that $ST\in\tilde{F}$. Then,
 \begin{equation*}
 \mathcal{R}\cup \sigma_{asc}(S+T)\setminus\{0\}= \mathcal{R}\cup (\sigma_{asc}(S)\cup \sigma_{asc}(T))\setminus\{0\}.
\end{equation*}
\end{theorem}
In the descent case we obtained the following results:
\begin{theorem}\label{(C)}
   Assume that $S$ and $T$ are two bounded linear operators satisfying $(H_1)$ and $(H_2)$.
 If $codim R(S^{n_0})$ and $codim R(T^{n_0})$ are finite then
$$
dsc(T)\leq n_0 \hbox{ and }dsc(S)\leq n_0 \hbox{ if and only if }dsc(TS)\leq n_0 .
$$
\end{theorem}
\begin{corollary}\label{nov}
 If $ST=TS\in\tilde{\mathcal{F}}$ then :\\
i) $\sigma_{dsc}(S+T)\setminus\{0\}\subseteq(\sigma_{dsc}(S)\cup \sigma_{dsc}(T))\setminus\{0\}.$\\
ii)
$$
\mathcal{M}\cup\mathcal{N}\cup\sigma_{dsc}(S+T)\setminus\{0\}=\mathcal{M}\cup\mathcal{N}\cup(\sigma_{dsc}(S)\cup \sigma_{dsc}(T))\setminus\{0\}.
$$
\end{corollary}

   In \cite{km}, the authors show that if $R$ is a regularity in a Banach algebra
   $\mathcal{A}$ then we have, if $a,\,b\in\mathcal{A}$, $ab=ba$ and $a\in Inv(\mathcal{A})$ then :
   \begin{equation}\label{mul}
   ab\in R \Leftrightarrow a\in R\hbox{ and } b\in R.
   \end{equation}
   In our work, dealing with ascent and descent operators as a special case of  a regularity, we prove in  Lemma \ref{theo3.4} and Theorem \ref{(C)} that
   property (\ref{mul}) is satisfied  without needing the invertibility of any of the bounded operators $T$ and $S$.
   
  On the other hand, inspired by the continuity concept of families of magnetic pseudo-differential operators given in \cite{nassimAJ}, we create a concept of convergence of spaces to prove condition $(C_2)$. Our main result in this context follow on :
  \begin{theorem}\label{T1}
Let $(T_n)_{n\in\mathbb{N}}$ be a sequence of bounded linear operators convergent to $T$ in the operator norm. Assume that $T$ has a closed range and for all $x$ in $N(T)$, $dist(x,N(T_n))$ is reached from some rank.
 Then:\\ i) $\lambda\in\sigma_{asc}(T)$ if and only if there exists a sequence  $\lambda_n\in\sigma_{asc}(T_n)$  convergent to $\lambda$.\\ii) $\lambda\in\sigma_{dsc}(T)$ if and only if there exists a sequence  $\lambda_n\in\sigma_{dsc}(T_n)$ convergent to $\lambda$.
\end{theorem}

  Theorem \ref{T1} is based on the fact that the limit operator $T$ has a closed range. We do not need neither the commutativity of the operators $(T_n)_n$ nor the fact that $T_n$ has a closed range for every $n\in\mathbb{N}$. It suffices that that the range of $(T_n)_n$ is closed from some $p\in\mathbb{N}$.

Finally , we illustrate our theoretical results by an application. It might be said that our approache throughout this paper is purely algebraic.  This is different from the approache used  in \cite{hocine, isole }, which is based on the use of analytic functions and the SVEP condition.\\
Note that , the following parts are devoted to proving the main results (Theorem \ref{th1}, Theorem \ref{(C)}, Coorollary \ref{nov} Theorem \ref{T1}) and  as well as the application.
\section{Proof of the main results :}
\subsection{On the theory of subspectrum}
\subsubsection{Results assuring condition $C_1$}
\paragraph{Theorem \ref{th1}} : First, we demonstrate the following lemmas.
\begin{lemma}\label{theo3.4}
   Let $T$ and $S$ be two bounded linear operators satisfying $(H_1)$. Then, $T$ and $S$ have  finite ascents if and only if
  $TS$ has a finite ascent.
\end{lemma}
\begin{proof}
Let $T$ and $S$ be respectively with finite ascents. Then, there exists $n_0\in\mathbb{N}$ such that for all $p\geq n_0$, we have
$$
\left\{
  \begin{array}{ll}
    T^px=0&\Rightarrow  T^{n_0}x=0 \\
& and\\
  S^px=0&\Rightarrow S^{n_0}x=0  .
  \end{array}
\right.
$$
On the other hand, we have $(TS)^px=0\Rightarrow T^p(S^px)=0\Rightarrow T^{n_0}(S^px)=0 \Rightarrow S^p(T^{n_0}x)=0\Rightarrow
S^{n_0}(T^{n_0}x)=0\Rightarrow (TS)^{n_0}x=0.
$ This  gives that $ TS$ is with finite ascent.
\par Concerning the inverse assertion, we have $TS$ is with finite ascent means that there exists $n_0\in\mathbb{N}$ such that for all $p\geq n_0$, we have :
\begin{equation}\label{charrouza}
  x\in N((TS)^p)\Rightarrow x\in N((TS)^{n_0}).
\end{equation}
Now, without loss of generality,  let $x\in N(S^p)$ such that $p\geq n_0$ and let us  show that $x\in N(S^{n_0})$. In fact, $x\in N(S^p)$ implies that
$x\in N((ST)^p)$. In view of (\ref{charrouza}), it yields that $x\in N((ST)^{n_0})$.
 Thus, by $(H_1)$, there exists $x_1\in N(S^{n_0})$ and $x_2\in N(T^{n_0})$ such that $x=x_1+x_2$. Since  $N(S^p)\subseteq N(S^{p+1})$
 and  $N(T^p)\subseteq N(T^{p+1})$ for all $p\in\mathbb{N}$ then $x_1\in N(S^p)$ and $x_2\in N(T^p)$. Furthermore, as $N(S^p)$ is a vector subspace and  $x_2=x-x_1$, it yields that $x_2\in N(S^p)\cap N(T^p)$. By  assumption $(H_1)$, we obtain  $x_2=0$. Hence, $x=x_1\in N(S^{n_0})$. This proves that $S$ is ascent.
\end{proof}
\begin{lemma}\label{monn}
 Let $S$ and $T$ be two bounded operators satisfying $ST=TS\in\tilde{\mathcal{F}}$ then :\\
 $$
  \sigma_{asc}(S+T)\setminus\{0\}\subseteq(\sigma_{asc}(S)\cup \sigma_{asc}(T))\setminus\{0\}.
 $$
\end{lemma}
  \begin{proof}
 Let $\lambda\neq0$ and $\lambda\in\rho_{asc}(S)\cap\rho_{asc}(T)$, then :
$$
\left\{
  \begin{array}{ll}
    (S-\lambda)&\in Asc(B(X)) \\
    &and \\
     (T-\lambda)&\in Asc(B(X)).
  \end{array}
\right.
$$
Since $ST=TS$, then
$$
(S-\lambda)(T-\lambda)=(T-\lambda)(S-\lambda).
$$
By the direct assertion of Lemma \ref{theo3.4} which is true for any two commutative operators, we obtain that
$$
(S-\lambda)(T-\lambda)\in Asc(B(X)).
$$
As,
\begin{equation}\label{hi}
(S-\lambda)(T-\lambda)=ST-\lambda(S+T-\lambda),
\end{equation}
it yields that,
\begin{equation}\label{som1}
  ST-\lambda(S+T-\lambda)\in Asc(B(X)).
\end{equation}
Since,
$$
ST(S+T-\lambda)=(S+T-\lambda)ST\quad \hbox{and}\quad ST\in\tilde{\mathcal{F}},
$$
we obtain  by    \cite[Theorem  2.2]{31-int}, in view of (\ref{som1}), that $(S+T-\lambda)\in Asc(B(X))$. Consequently,
\begin{equation}\label{eq3.8}
  \sigma_{asc}(S+T)\setminus\{0\}\subseteq(\sigma_{asc}(S)\cup\sigma_{asc}(T))\setminus\{0\}.
\end{equation}
\end{proof}
\paragraph*{Proof of Theorem \ref{th1}}
The direct inclusion follows from Lemma \ref{monn}. To prove the reciprocal inclusion, let $\lambda\in
(\sigma_{asc}(S)\cup \sigma_{asc}(T))\setminus\{0\}$ and assume that $\lambda\notin \sigma_{asc}(S+T)\setminus\{0\}$ and $\lambda\notin\mathcal{R}$. Then, $S+T-\lambda\in Asc(B(X))$ and $(S-\lambda)$ and $(T-\lambda)$ satisfy $(H_1)$. Since, $ST=TS\in\tilde{\mathcal{F}}$, then by
  \cite[Theorem  2.2]{31-int}, we obtain  in view of (\ref{hi}) that :
$$
(S-\lambda)(T-\lambda)=(T-\lambda)(S-\lambda)\in Asc(B(X)).
$$
 As $(S-\lambda)$ and $(T-\lambda)$ satisfy $(H_1)$, it follows from Lemma \ref{theo3.4} that $(S-\lambda)$ and
 $(T-\lambda)$ are ascent, which is absurd. This proves the second inclusion.\qed
\paragraph{Theorem \ref{(C)}} :  In order to prove Theorem  \ref{(C)}, we first prove some auxilary resuls.
Denote by  $\mathcal{P}(X)$ the set of all projections $P\in B(X)$ such that $codim R(P)$ is finite. For $T\in B(X)$ and $P\in\mathcal{P}(X)$, the compression $T_P :R(P)\to R(P)$ is defined by $T_P(y)=PTy$, $y\in R(P)$,
i.e. $T_P=PT_{|R(P)}$ where $T_{|R(P)} : R(P)\to X$ is the restriction of $T$. Clearly, $R(P)$ is a Banach space and $T_P\in B(R(P)))$. 
\begin{lemma}\label{4.1} 
Let $T\in B(X)$ and $X$ be a direct sum of closed subspaces $X_1$ and $X_2$ which are $T-$invariant. If
$T_1=T_{|X_1} : X_1\to X_1$ and $T_2=T_{|X_2} : X_2\to X_2$ then the following statements hold true :\\
i)  $T$ has a finite  ascent if and only if $ T_1$ and $T_2$ have respectively finite ascents.\\
ii)  $T$ has a finite  descent if and only if $ T_1$ and $T_2$ have respectively finite descents.
\end{lemma}
We prove the next result analogously as in \cite{zdr} :
\begin{lemma}\label{ca}
  For $T\in B(X)$ and $P\in\mathcal{P}(X)$, the following assertions hold :\\
 i) If $TP=PT$ then $T\in Asc(B(X))$ if and only if $T_P\in Asc(B(X))$.\\
 ii) If $TP=PT$ then $T\in Dsc(B(X))$ if and only if $T_P\in Dsc(B(X))$.
\end{lemma}
\begin{proof}
  i) Assume that $T\in B(X)$, $P\in\mathcal{P}(X)$ and $TP=PT$. Then, $X=R(P)\oplus N(P)$ and the subspaces $R(P)$ and $N(P)$ are invariant by $PTP\in B(X)$. The operator $PTP$ has the following matrix form :
  \begin{equation}\label{repre}
    PTP=\left(
          \begin{array}{cc}
            T_p & 0 \\
            0 & 0 \\
          \end{array}
        \right):\left(
                  \begin{array}{c}
                    R(P) \\
                    N(P) \\
                  \end{array}
                \right)\to\left(
                  \begin{array}{c}
                    R(P) \\
                    N(P) \\
                  \end{array}
                \right).
  \end{equation}
  From Lemma \ref{4.1} i) $($respectively, ii)$)$, it yields that $PTP$ is with finite ascent $($respectively, descent $)$ if and only if $T_P$ is
  with finite ascent $($respectively, descent $)$. Since,
  $$
  T=PT+(I-P)T=PTP+PT(I-P)+(I-P)T,
  $$
  and $PT(I-P)+(I-P)T$ is a finite rank operator commuting with $PTP$ it yields by \cite[ Theorem  2.2]{31-int}
  (respectively, \cite[Theorem 3.1]{burgos}), that $PTP$ is  with finite ascent (respectively. descent) if and only if  $T$ is  with finite ascent
  (respectively. descent).
  \end{proof}

.
\begin{lemma}\label{lemma3.5}
  Let $T$ and $S$ be two bounded linear operators such that $TS=ST$. Then,
  $$ \hbox{ $T$ and $S$ have  finite descents}\Rightarrow
  TS \hbox{ has a finite descent. }
  $$
\end{lemma}
\begin{proof} Let $T$ and $S$ be respectively with finite descents. Then there exists $n_0\in\mathbb{N}$ such that for all $p\geq n_0$ we have
$$
\left\{
  \begin{array}{ll}
     T^p(X)&=T^{p+1}(X) \\
    & and \\
    S^p(X)&=S^{p+1}(X).
  \end{array}
\right.
$$
Furthermore, we have
  $$
  (ST)^p(X)=S^pT^p(X)=S^pT^{p+1}(X)=S^{p+1}T^{p+1}(X)=(ST)^{p+1}(X).
  $$
  Consequently, $ST$ has a finite descent.
  \end{proof}
\begin{lemma}\label{lemma3.6}
   Let $T$ and $S$ be two bounded linear operators satisfying $(H_1)$ and  $(H_2)$. Then, $T$ or $S$ has a finite descent.
    \end{lemma}
\begin{proof}
  Since $dsc(TS)=n_0$ is finite, we obtain by Proposition
  \ref{cont} that for every $x\in X$ there exists $w\in R(TS)$ and $x_2\in N((TS)^{n_0})$ satisfying :
   $$
   x=w+x_2.
   $$
  In other words, there exists $x_1\in X$ such that $w=TSx_1$ and $T^{n_0}S^{n_0}x_2=0$. By $(H_1)$, there exists $x_2'\in
  N(S^{n_0})$ and $x_2''\in N(T^{n_0})$  satisfying
   $$
   x=(TS)x_1+x_2'+x_2''.
   $$
   Without loss of generality, we assume that $N(S^{n_0})\subseteq R(T)$. Hence, there exists $y_2\in X$ such that $x_2'=Ty_2$. Put $y_1:=Sx_1$,
   it follows that
 $$
 x=T(y_1+y_2)+x_2''.
 $$
The result  follows from Proposition \ref{cont}.
  \end{proof}
\paragraph*{Proof of Theorem \ref{(C)}}
  Concerning the direct sense, it follows from Lemma \ref{lemma3.5}.\\
   Now, concerning the reciprocal sense, we obtain from Lemma \ref{lemma3.6} that $dsc(T)\leq n_0$ or $dsc(S)\leq n_0$.
Without loss of generality, we assume that $dsc(S)\leq n_0$. Using the fact that $TS$ has a finite descent, we obtain that for all $n\geq n_0$,
$$
(TS)^nX=(TS)^{n+1}X, \quad S^nX=S^{n+1}X=S^{n_0}X, \hbox{ and } TS=ST.
$$
This means that for all $n\geq n_0$,
\begin{equation}\label{33}
  T^nS^{n_0}X= T^{n+1}S^{n_0}X.
\end{equation}
Consider $P$, the project operator on $R(S^{n_0})$. One can remark that $codim R(P)<\infty$. Then, $P\in\mathcal{P}(X)$. Let $T_{P} : R(S^{n_0})\to
R(S^{n_0})$, defined by $T_{P} (y)=PT(y)$. Note that $TP=PT$  then, for all $n\geq n_0$, we have in view of (\ref{33}),
$$
T^n_P(X)=(TS^{n_0})^nX=T^{n+1}S^{n_0}X=(TS^{n_0})^{n+1}X=T_P^{n+1}(X),
$$
which means that $T_P$ has a finite descent. Hence, using Theorem \ref{ca}, $dsc(T)$ is finite.\qed
\paragraph{Proof of Corollary \ref{nov} :}
  i) Let $\lambda\neq0$ and $\lambda\in\rho_{dsc}(S)\cap\rho_{dsc}(T)$, then :
$$
\left\{
  \begin{array}{ll}
    (S-\lambda)&\in Dsc(B(X)) \\
    &and \\
     (T-\lambda)&\in Dsc(B(X)).
  \end{array}
\right.
$$
Since $ST=TS$, then
$$
(S-\lambda)(T-\lambda)=(T-\lambda)(S-\lambda).
$$
Using Lemma \ref{lemma3.5}, we obtain
$$
(S-\lambda)(T-\lambda)\in Dsc(B(X)).
$$
Since,
\begin{equation}\label{34}
(S-\lambda)(T-\lambda)=ST-\lambda(S+T-\lambda),
\end{equation}
we have,
\begin{equation}\label{somh}
  ST-\lambda(S+T-\lambda)\in Dsc(B(X)).
\end{equation}
Remark that,
$$
ST(S+T-\lambda)=(S+T-\lambda)ST\quad \hbox{and}\quad ST\in\tilde{\mathcal{F}}.
$$
It follows by \cite[Theorem 3.1]{burgos} and (\ref{somh}) that
$(S+T-\lambda)\in Dsc(B(X))$. Hence,
\begin{equation}\label{eq3.8}
  \sigma_{dsc}(S+T)\setminus\{0\}\subseteq(\sigma_{dsc}(S)\cup\sigma_{dsc}(T))\setminus\{0\}.
\end{equation}
ii) According to i) , we have :
\begin{equation}\label{premi}
\mathcal{M}\cup\mathcal{N}\cup \sigma_{dsc}(S+T)\setminus\{0\}\subseteq \mathcal{M}\cup\mathcal{N}\cup(\sigma_{dsc}(S)\cup
\sigma_{dsc}(T))\setminus\{0\}.
\end{equation}
Now, concerning the inverse inclusion, let $\lambda\in(\sigma_{dsc}(S)\cup\sigma_{dsc}(T))\setminus\{0\}$ and assume that $\lambda\notin
\sigma_{dsc}(S+T)\setminus\{0\}$, $\lambda\notin \mathcal{M}$ and $\lambda\notin \mathcal{N}$. Then, $S+T-\lambda\in Dsc(B(X))$. Since
$ST=TS\in\tilde{\mathcal{F}}$, then by \cite[Theorem 3.1]{burgos} and (\ref{34}), we obtain
$$
(S-\lambda)(T-\lambda)=(T-\lambda)(S-\lambda)\in Dsc(B(X)).
$$
On the other hand, since $\lambda\notin \mathcal{M}$ then $dsc((S-\lambda)(T-\lambda))=n_0$,
$codim(R(T-\lambda)^{n_0})$ and $codim(R(S-\lambda)^{n_0})$ are finite. Using the fact  that $(S+T-\lambda)\in Dsc(B(X))$ and $\lambda\notin\mathcal{N}$, it follows  that $(S-\lambda)$ and $(T-\lambda)$ satisfy $(H_1)$ and $(H_2)$. Hence,
 by Lemma \ref{(C)}, we obtain $dsc(S-\lambda)$ and $dsc(T-\lambda)$ are finite, which is absurd. Thus,
\begin{equation}\label{deusi}
  \mathcal{M}\cup\mathcal{N}\cup(\sigma_{dsc}(S)\cup\sigma_{dsc}(T))\setminus\{0\}\subseteq\mathcal{M}
  \cup\mathcal{N}\cup\sigma_{dsc}(S+T)\setminus\{0\}.
\end{equation}
The result follows from (\ref{premi}) and (\ref{deusi}). \qed
 \subsection{Results assuring condition $(C_2)$}
We, first, give the following concept of convergence of spaces.
\begin{definition}\label{def}
  Let $(E_n)_n$ be a sequence of normed subspaces of $X$.\\
i)  We say that $(E_n)_n$  upper-converges to a vector space $E$, and we write $u-\lim_{n\rightarrow\infty}E_n=E$, if $x$ is an adherent point  of a sequence $(x_n)_n\subset X$ such that $x_n\in E_n$ from some rank implies that $x\in E$.\\
ii)  We say that $(E_n)_n$  lower-converges to a vector space $E$ and we write $l-\lim_{n\rightarrow\infty}E_n=E$, if $x$ belongs to $E$ implies that $x$ is an adherent point of a sequence $(x_n)_n\subset X$ such that $x_n\in E_n$ from some rank.
\end{definition}
Let $X$ be a Banach space and let $T : X\rightarrow X$ be a non zero operator. We define the reduced minimum modulus of $T$ by
 $$
\gamma(T):=\inf\{\|Tx\|,; x\in X, dist (x, N(T))=1\}.
$$
Formally, we set $\gamma(0)=\infty$.\\
Let $Y$, $Z$ be subspaces of $X$ and define
$$
\delta(Y,Z) :=\sup_{x\in Y\\ \|x\|\leq 1}dist(x,Z).
$$
The gap $\widehat{\delta}(Y,Z)$ is defined by $\widehat{\delta}(Y,Z)=\max(\delta(Y,Z),\delta(Z,Y))$.\\
 To prove Theorem \ref{T1}, the main result of this subsection we need to prove next lemmas and propositions. Note that Proposition \ref{lem2},  Proposition \ref{lem1},   Proposition \ref{lem4} and Proposition \ref{lem1} summarize the resultat of Theorem \ref{T1}.
\begin{lemma}\label{lemmasup}
 i) Let $(E_n)_n$ be a sequence of normed subspaces of a Banach space $X$ , upper-convergent to a normed vector space $E$. Then $\delta(E_n,E)$ converges to $0$.\\
 ii) Let $(E_n)_n$ be a sequence of normed subspaces of a Banach sapace $X$, lower-convergent to a normed vector space $E$. Then $\delta(E,E_n)$ converges to $0$.\\
 iii) Let $(E_n)_n$  be a sequence of Banach subspaces, upper-convergent to  $\{0\}$. Then $E_n=\{0\}$ from some rank.\\
 iv) Let $(E_n)_n$  be a sequence of Banach subspaces, lower-convergent to $X$. Then $E_n=X$ from some rank.
\end{lemma}
\begin{proof}
It is easy to prove i) and ii). Concerning iii) and iv), it suffices to use i), ii) as well as Theorem 17 page 102 in \cite{Muller}.
\end{proof}
\begin{lemma}\label{lemma1}
  Let $(T_n)_n$ be a sequence of bounded linear operators convergent to $T$ in the operator norm. Then, $(N(T_n))_n$ upper-converges to $N(T)$.
\end{lemma}
\begin{proof} Let $(x_n)_n\subset X$ such that $x_n\in N(T_n)$ from  some rank. Next, we  prove that for every adherent point $x$ of $(x_n)_n$, we have $x\in N(T)$. Indeed, for every $n\in\mathbb{N}$
\begin{eqnarray}\label{chi}
  \|Tx\| &=& \|Tx+Tx_n-Tx_n+T_nx_n-T_nx_n\| \nonumber\\
   &\leq& \|T\|\|x-x_n\|+\|(T-T_n)x_n\|+\|T_nx_n\|.
\end{eqnarray}
 Since $x_n\in N(T_n)$ from some rank $n_0\in\mathbb{N}$, then  $\|T_nx_n\|=0$, for all $n\geq n_0$. Let $\varepsilon>0$, for all  $N\in\mathbb{N}$ there exists $n\geq N$, such that $\|x_n-x\|<\varepsilon$.   Besides, there is $n_1\in\mathbb{N}$ such that for every $n\geq n_1$, we have $\|(T-T_n)x\|<\varepsilon$ for all $x\in X$. Thus, we obtain by (\ref{chi}) that for all $N_0=\sup(n_0, n_1)$ there exists $n\geq N_0$ satisfying $\|Tx\|<(\|T\|+1)\varepsilon$. Consequently, $x\in N(T)$.
\end{proof}
\begin{lemma}\label{lemma 5}
  Let $(T_n)_n$ be a sequence of bounded linear operators convergent to $T$ in the operator norm. Assume that $T$ has a closed range and for all $x$ in $N(T)$, $dist(x,N(T_n))$ is reached for some rank, then $(N(T_n))_n$ lower-converges to $N(T)$.
\end{lemma}
\begin{proof}
  By Lemma \ref{lemmasup} i) and  Lemma \ref{lemma1}, $\delta(N(T_n), N(T))$ tends to $0$ as $n\rightarrow\infty$. Using  Theorem 17 page 102 in \cite{Muller}, we have that if $T$ is with closed range then $\delta(N(T), N(T_n))$ tends to $0$ as $n\rightarrow\infty$. This implies that
$$
\sup\limits_{x\in N(T),\ \|x\|\leq1}\inf\limits_{y\in N(T_n)}\|x-y\| \hbox{ tends to $0$ as }n\rightarrow\infty,
$$
 which implies that
 $$
 \forall x\in N(T), \|x\|\leq 1, \inf\limits_{y\in N(T_n)}\|x-y\| \hbox{ tends to $0$ as $n\rightarrow\infty$. }
 $$
  That is, for all $x\in N(T)$ and $\|x\|\leq 1$,  there exists $n_0\in\mathbb{N}$ such that for all $n\geq n_0$ there is $y_n\in N(T_n)$ satisfying $\|x-y_n\|=\inf\limits_{y\in N(T_n)}\|x-y\|$.
\end{proof}
\begin{lemma}\label{lemma3}
   Let $(T_n)_n$ be a sequence of bounded linear operators convergent to $T$ in the operator norm. Assume that $\limsup\gamma(T_n)>0$. Then, $(R(T_n))_n$ upper-converges to $R(T)$.
\end{lemma}
\begin{proof} See Corollary 19 page 103 in \cite{Muller}.\end{proof}
\begin{lemma}\label{lemma3'}
   Let $(T_n)_n$ be a sequence of bounded linear operators convergent to $T$ in the operaor norm. Then, $(R(T_n))_n$ lower-converges to $R(T)$.
\end{lemma}
\begin{proof} Consider,
$$
F=\left\{(y_n)_n\subset X \hbox{ such that }y_n\in R(T_n) \hbox{ from some rank} \right\}.
$$
 Assume that $y\in R(T)$. We will prove that $y$ is a limit of a sequence $(y_n)_n$ belonging to $F$. In fact, $y\in R(T)$ means that there exists $x\in X$ satisfying $Tx=y$.
Consider the sequence $y_n=T_nx$, from some rank $n_0\in\mathbb{N}$. Since $T_n$ converges to $T$ in the operator norm , then for all $\varepsilon>0$, there exists $n_1\in\mathbb{N}$ such that every $n\geq n_1$, $\|Tx-T_nx\|<\varepsilon$. That is $\|y-y_n\|<\varepsilon$. Hence, $(y_n)_n$ converges to $y$ and $(y_n)_n\in F$.
\end{proof}
\begin{proposition}\label{lem2}Let $(T_n)_{n\in\mathbb{N}}$ be a sequence of bounded linear operators convergent to $T$ in the operator norm. Assume that $T$ has a closed range and for all $x$ in $N(T)$, $dist(x,N(T_n))$ is reached from some rank. If $(\lambda_n)_n$ is a sequence convergent to $\lambda\in\sigma_{asc}(T)$ then $\lambda_n\in \sigma_{asc}(T_n)$ from some rank.
\end{proposition}
\begin{proof}
Let $\lambda\in\sigma_{asc}(T)$. By proposition \ref{cont}, this is equivalent to say:
\begin{equation}\label{etoil}
\hbox{ For all $m\geq 0$, there exists $d \geq m$ satisfying $R((T-\lambda)^d)\cap N(T-\lambda)\neq\{0\}$}.
\end{equation}
Since  $T_n-\lambda_n$ converges to $T-\lambda$ in the operator norm, then
by Lemma \ref{lemma 5} $($respectively Lemma \ref{lemma3'}$)$, $R(T-\lambda)^d$ $($respectively $N(T-\lambda)$$)$ is the $l-limit$ of the sequence $R(T_n-\lambda_n)^d$ $($respectively  $N(T_n-\lambda_n)$$)$. Thus, (\ref{etoil}) is equivalent to say that,
\begin{equation}\label{en-n}
  l-\lim\limits_{n\to\infty}\left( N(T_n-\lambda_n) \cap R((T_n-\lambda_n)^d)\right) \neq\{0\}.
\end{equation}
From (\ref{en-n}) and Lemma \ref{lemmasup} ii),
$$
  \hbox{ for every $m\geq 0$ there exists $d\geq m$ satisfying } N(T_n-\lambda_n) \cap R((T_n-\lambda_n)^d) \neq\{0\},\hbox{  from some rank $n\in\mathbb{N}$ }.
$$
Therefore, according to Proposition \ref{cont}, $\lambda_n\in\sigma_{asc}(T_n),$ from some rank, .
\end{proof}
\begin{proposition}\label{lem1}
  Let $(T_n)_{n\in\mathbb{N}}$ be a sequence of bounded linear operators  convergent to $T$ in the operator norm. Assume that for every $n\in\mathbb{N}$,  $\lambda_n\in\sigma_{asc}(T_n)$ and $\limsup\gamma(T_n)>0$. If $(\lambda_n)_n$ converges to $\lambda$ then $\lambda\in\sigma_{asc}(T)$.
\end{proposition}
\begin{proof}
Let $\lambda_n\in\sigma_{asc}(T_n)$. This means, in view of Proposition \ref{cont}, that from some rank $n\in\mathbb{N}$,
 \begin{equation}\label{fleur}
\hbox{ for all $m\geq 0$, there exists $d\geq m$ satisfying } R((T_n-\lambda_n)^d)\cap N(T_n-\lambda_n)\neq \{0\},
 \end{equation}
Since $(T_n-\lambda_n)_n$ converges to $(T-\lambda)$ in the operator norm, then by Lemma \ref{lemma1} and  Lemma \ref{lemma3},  $N(T_n-\lambda_n)$ upper-converges to $N(T-\lambda)$ and
$R(T_n-\lambda_n)^d$ upper-converges to $R((T-\lambda)^d)$. Using Lemma \ref{lemmasup} iii),  (\ref{fleur}) implies,
$$
 \hbox{for all $m\geq 0$ there exists $ d\geq m$ satisfying } R((T-\lambda)^d)\cap N(T-\lambda)\neq \{0\}.
$$
Consequently, by Proposition \ref{cont}, $asc(T-\lambda)$ is infinite. Hence, $\lambda\in\sigma_{asc}(T)$.
\end{proof}
In order, to deal with the case of descent spectrum, we need to pove the following results:
\begin{proposition}\label{lem4}
 Let $(T_n)_{n\in\mathbb{N}}$ be a sequence of bounded linear operators convergent to $T$ in the operator norm. Assume that $\limsup\gamma(T_n)>0$.
 If $(\lambda_n)_n$ is convergent to $\lambda\in\sigma_{dsc}(T)$ then $\lambda_n\in \sigma_{dsc}(T_n)$ from some rank.
\end{proposition}
\begin{proof}
Let $\lambda\in\sigma_{dsc}(T)$.  Using Proposition \ref{cont}, we have
\begin{equation}\label{V}
\hbox{ for all  $m\geq 0$, there exists $d \geq m$ satisfying } R(T-\lambda)+ N((T-\lambda)^d)\subsetneq X.
\end{equation}
  Since $(T_n-\lambda_n)_n$ converges to $(T-\lambda)$, then by Lemma \ref{lemma 5} (respectively,  Lemma \ref{lemma3'}),  $N(T-\lambda)$ (respectively, $R(T-\lambda)$) is the $u-limit$ of the sequence $(N(T_n-\lambda_n))_n$ (respectively, $(R(T_n-\lambda_n))_n$). Thus, (\ref{V}) is equivalent to say
$$
 \hbox{ for every $m\geq 0$ there exists $d\geq m$ satisfying }u-\lim\limits_{n\to\infty}\left( R(T_n-\lambda_n)+ N((T_n-\lambda_n)^d)\right)\subsetneq X.
$$
Using Lemma \ref{lemmasup} i) it follows,
\begin{equation}\label{lambda-n}
 \hbox{ from some rank $n\in\mathbb{N}$, for every $m\geq 0$, there exists $d\geq m$ satisfying } R(T_n-\lambda_n)+ N((T_n-\lambda_n)^d)\subsetneq X.
\end{equation}
Hence, $\lambda_n\in\sigma_{dsc}(T_n)$, from some rank.
\end{proof}
The following result is the reciprocal of Propostion \ref{lem4}.
\begin{proposition}\label{lem3}
  Let $(T_n)_{n\in\mathbb{N}}$  be a sequence of bounded linear operators convergent to $T$ in the operator norm. Assume that for every $n\in\mathbb{N}$, $\lambda_n\in\sigma_{dsc}(T_n)$, that $T$ has a closed range and  that for all $x$ in $N(T),$ $dist(x,N(T_n))$ is reached from some rank. If $\lambda_n$ converges to $\lambda$ then $\lambda\in\sigma_{dsc}(T)$.
\end{proposition}
\begin{proof}
Let $\lambda_n\in\sigma_{dsc}(T_n)$ for every $n\in\mathbb{N}$. This means by Proposition \ref{cont},
\begin{equation}\label{4.16}
\hbox{ for all $m\geq0$, there exists $d\geq m$  satisfying  }R(T_n-\lambda_n)+N((T_n-\lambda)^d)\subsetneq X.
\end{equation}
On the other hand,  $T_n-\lambda_n$ converges to $T-\lambda$ in the operator norm.  By Lemma \ref{lemma1} and Lemma \ref{lemma3}, $N(T_n-\lambda_n)^d$ lower-converges to $N(T-\lambda)^d$ and  $R(T_n-\lambda_n)^d$ lower-converges to $R(T-\lambda)$. Thus, it follows by Lemma \ref{lemmasup} iv) that (\ref{4.16}) implies
$$
 \hbox{ for all $m\geq 0$ there exists $ d\geq m$ satisfying }R(T-\lambda)+N((T-\lambda)^d)\subsetneq X.
$$
Hence, in view of Proposition \ref{cont}, $\lambda\in\sigma_{dsc}(T)$.
\end{proof}
\paragraph*{Proof of Theorem \ref{T1}.}
  
According to Proposition \ref{lem4} and Proposition \ref{lem3},  we obtained i), thus ii) was  is immediate consequence of   Proposition \ref{lem2} and Proposition \ref{lem1} \qed
\section{Application}
Let $\mathcal{H}_{1}$ and $\mathcal{H}_{2}$ be two Hilbert space. We consider the two $2\times2$ block operator matrices defined on $\mathcal{H}_{1}\times\mathcal{H}_{2}$ by
\begin{align*}M=\left(
 \begin{array}{cc}
 T & 0 \\
 0 & S \\
\end{array}
\right)\in B(\mathcal{H}_{1}\times\mathcal{H}_{2})
\hbox{ and }
M_{C}=\left(
 \begin{array}{cc}
 T & C \\
 0 & S \\
\end{array}
\right)\in B(\mathcal{H}_{1}\times\mathcal{H}_{2}),\\
\end{align*}
\noindent where,  $T\in B(\mathcal{H}_{1})$,
  $S \in B(\mathcal{H}_{2})$ and $C\in B(\mathcal{H}_{2},\mathcal{H}_{1})$. Observe that
  \begin{align*}\displaystyle M_{\frac{1}{k}C}=\left(
 \begin{array}{cc}
 I & 0 \\
 0 & kI \\
\end{array}
\right)
\left(
 \begin{array}{cc}
 T & C \\
 0 & S \\
\end{array}
\right)
\left(
 \begin{array}{cc}
 I & 0 \\
 0 & \frac{1}{k}I \\
\end{array}
\right), \mbox{ for }k\in\mathbb{N}^{*}.
\end{align*}
\noindent We assume that $S$ and $T$ have  closed ranges. Since $M_C$ and $M_{\frac{1}{k}C}$ are similar, it follows that $\sigma_i(M_{\frac{1}{k}C})=\sigma_i(M_C)$, $i\in\{asc, dsc\}$.\\\\
Let $\mathcal{T}$ and $\mathcal{S}$ be two $2\times2$ block operator matrices defined by
 \begin{align*}\displaystyle \mathcal{T}=\left(
\begin{array}{cc}
 T & 0 \\
 0 & 0 \\
\end{array}
\right)\in B(\mathcal{H}_{1}\times\mathcal{H}_{2})
\hbox{ and }
\mathcal{S}=\left(
 \begin{array}{cc}
 0 & 0 \\
 0 & S \\
\end{array}
\right)\in B(\mathcal{H}_{1}\times\mathcal{H}_{2}).\\
\end{align*}It is easy to verify  that hypothesis of Corollary \ref{nov} and Lemma \ref{monn} are satisfied. Consequently,  $\sigma_i(M)\backslash\{0\}\subset(\sigma_i(\mathcal{T})\cup\sigma_i(\mathcal{S}))\backslash\{0\}$.
Using Lemma \ref{4.1}, $\sigma_i(M)=\sigma_i(T)\cup\sigma_i(S)=\sigma_i(\mathcal{T})\cup\sigma_i(\mathcal{S})$.
 By Theorem \ref{T1}, $\sigma_i(M_C)=\sigma_i(T)\cup\sigma_i(S)$.

\end{document}